\documentclass[12pt]{amsart}

\usepackage{amssymb, amscd, txfonts}
\usepackage{graphicx}
\usepackage[all]{xy}  
 
\numberwithin{equation}{section}

\newtheorem{theorem}{Theorem}[section]
\newtheorem{proposition}[theorem]{Proposition}
\newtheorem{lemma}[theorem]{Lemma}
\newtheorem{corollary}[theorem]{Corollary}

\theoremstyle{definition}

\theoremstyle{remark}
\newtheorem{remark}[theorem]{Remark}
\newtheorem{claim}[theorem]{Claim}
\newtheorem{step}{Step}
\newtheorem{step+}{Step}

\newcommand{\Z}{\mathbb{Z}}
\newcommand{\Q}{\mathbb{Q}}

\newcommand{\LL}{\Lambda}
\newcommand{\LQ}{\Lambda_{\mathbb{Q}}}
\newcommand{\Lv}{\Lambda^{\vee}}
\newcommand{\DL}{D_{\Lambda}}
\newcommand{\DLp}{D_{\Lambda, p}}
\newcommand{\GL}{\Gamma_{\Lambda}}
\newcommand{\SpL}{{\rm Sp}(\Lambda)}
\newcommand{\dv}{{\rm div}}

\begin{document}

\title[]{Symplectic Eichler criterion}
\author[]{Shouhei Ma}
\thanks{Supported by KAKENHI 21H00971} 
\address{Department~of~Mathematics, Science~Institute~of~Tokyo, Tokyo 152-8551, Japan}
\email{ma@math.titech.ac.jp}
\keywords{} 

\begin{abstract}
We prove an Eichler-type criterion for symplectic lattices which determines in a simple way 
when two primitive vectors are equivalent under a canonical congruence subgroup of the symplectic group. 
This is supplemented by another, related criterion which determines when a given primitive vector is splitting. 
Both criteria use the discriminant groups. 
\end{abstract}

\maketitle

\section{Introduction}\label{sec: intro}

The \textit{Eichler criterion} is a simple criterion which determines 
when two primitive vectors of a quadratic lattice containing two hyperbolic planes are equivalent 
under a canonical subgroup of the orthogonal group. 
It goes back to Eichler (\cite{Ei} Satz 10.4), and was later given the modern formulation 
by various authors (e.g., \cite{Sca} Proposition 3.7.3, \cite{FH} Lemma 4.4, \cite{GHS} Proposition 3.3). 
Nowadays this criterion serves as a useful tool in the study of orthogonal modular varieties. 
In this note we give a symplectic analogue of the Eichler criterion. 

Let ${\LL}$ be a symplectic lattice, namely 
a free ${\Z}$-module of rank $2g$ endowed with a nondegenerate alternating form 
$( \cdot , \cdot ) \colon {\LL} \times {\LL} \to {\Z}$. 
It is well-known that ${\LL}$ is isometric to  $U(d_1) \oplus \cdots \oplus U(d_{g})$ 
for some natural numbers $d_{1}, \cdots, d_{g}$ with $d_i | d_{i+1}$ for all $i$, 
where $U$ is the rank $2$ unimodular symplectic lattice expressed by the matrix 
$\begin{pmatrix} 0 & 1 \\ -1 & 0 \end{pmatrix}$ 
and $U(d)$ is the $d$-scaling of $U$. 
The sequence  $(d_1, \cdots , d_g)$ is called the \textit{type} of ${\LL}$. 

The symplectic lattice ${\LL}$ is naturally embedded in its dual lattice ${\LL}^{\vee}$. 
The quotient group 
\begin{equation*}
{\DL}={\LL}^{\vee}/{\LL}\simeq ({\Z}/d_1)^{\oplus 2}\oplus \cdots  \oplus ({\Z}/d_g)^{\oplus 2}
\end{equation*} 
is called the \textit{discriminant group} of ${\LL}$ 
and endowed with a natural finite symplectic form.  
Let ${\GL}$ be the kernel of the reduction map ${\SpL}\to {\rm Sp}({\DL})$. 
This is a canonical congruence subgroup of ${\SpL}$. 

We are interested in when two primitive vectors of ${\LL}$ are ${\GL}$-equivalent. 
Let $v$ be a primitive vector of ${\LL}$. 
We denote by ${\dv}(v)$ the natural number with $(v, {\LL})= {\dv}(v){\Z}$. 
Then let $v^{\ast}=v/{\dv}(v)\in {\LL}^{\vee}$ and $[v^{\ast}]\in {\DL}$ be its class in ${\DL}$.  
Our result is the following. 

\begin{theorem}[Theorem \ref{thm: main}]\label{thm: main intro}
Assume that $d_1=d_2=1$. 
Then the assignment $v\mapsto [v^{\ast}]$ defines a bijective map 
from the set of ${\GL}$-equivalence classes of primitive vectors of ${\LL}$ 
to the set ${\DL}$. 
\end{theorem} 

Since the correspondence $v\mapsto [v^{\ast}]$ is canonical, 
Theorem \ref{thm: main intro} reduces classification of $\Gamma$-orbits 
for any group $\Gamma$ between ${\GL}$ and ${\SpL}$ 
to calculation of finite group action on ${\DL}$. 

Several authors have given criteria for equivalence of primitive vectors. 
James \cite{Ja} proposed a criterion for ${\SpL}$-equivalence. 
His ``invariants'' are, however, not unique. 
It is possible to single out a canonical one by imposing a maximality condition, 
but no effective procedure is known for achieving this 
(see ``Note added in proof'' in \cite{Ja}). 
Schellhammer (\cite{Sch} \S 4) gave an effective criterion for both ${\GL}$- and ${\SpL}$-equivalence 
in terms of the coordinates of vectors. 
His criterion looks somewhat close to the Eichler criterion, but not quite the same, 
as he works modulo natural number which depends on the coordinates of the vector. 
Compared to these predecessors, 
the advantage of our Eichler criterion would lie in the simplicity and canonicity, 
under the assumption $d_1=d_2=1$. 
 
There are certainly symplectic lattices with $d_2\ne 1$ for which the Eichler criterion still holds. 
A result of Hulek-Kahn-Weintraub (\cite{HKW} Proposition I.3.38) can be interpreted as saying that 
the Eichler criterion holds for $U\oplus U(p)$ with $p$ prime. 
We extend this result to 
$U\oplus U(p) \oplus \cdots \oplus U(p)$ 
(Proposition \ref{prop: (1,p,p)}). 
Friedland-Sankaran \cite{FS} studied the lattices $U\oplus U(t)$ with $t$ square-free. 
It is implicit in their result that the Eichler criterion still holds in this case. 

From a more wide point of view, Theorem \ref{thm: main intro} can be regarded as an instance 
demonstrating the utility of the symplectic discriminant forms. 
Both quadratic and symplectic discriminant forms were considered classically by Wall \cite{Wa}. 
While the theory of quadratic discriminant forms was developed systematically by Nikulin \cite{Ni} 
and has found many applications,  
the symplectic ones seem to have not been paid much attention. 

Our second purpose in this note is to refine the Eichler criterion in the following aspect. 
We are interested in when a primitive vector $v\in {\LL}$ is \textit{splitting}, i.e., 
a part of a rank $2$ sublattice ${\LL}'$ with ${\LL}={\LL}'\oplus ({\LL}')^{\perp}$. 
We call an element $x$ of ${\DL}$ of order $d$ \textit{splitting} if 
it is contained in a subgroup $D'\simeq ({\Z}/d)^{\oplus 2}$ with 
${\DL}=D'\oplus (D')^{\perp}$. 
We prove the following. 

\begin{theorem}[Theorem \ref{thm: splitting}]\label{thm: splitting intro} 
Assume that $d_1=d_2=1$. 
A primitive vector $v\in {\LL}$ is splitting if and only if the element $[v^{\ast}]\in {\DL}$ is splitting. 
\end{theorem}    

This reduces classification of splitting vectors of ${\LL}$ to that of splitting elements of ${\DL}$. 
The latter is much easier (\S \ref{ssec: classify split}). 

Theorem \ref{thm: main intro} and Theorem \ref{thm: splitting intro} 
will be used in our study of rank $1$ degeneration of universal abelian varieties. 
In fact, this was our motivation to establish the symplectic Eichler criterion.

\section{The Eichler criterion}

Let us begin by fixing the notation and terminology. 
Let ${\LL}$ be a symplectic lattice. 
We identify the dual lattice of ${\LL}$ with the sublattice 
\begin{equation*}
{\Lv} = \{ \: v\in {\LQ} \: | \: (v, {\LL})\subset {\Z} \: \}
\end{equation*}
of ${\LQ}$. 
Then ${\LL}\subset {\Lv}$, 
and the quotient group ${\DL}={\Lv}/{\LL}$ is called the \textit{discriminant group} of ${\LL}$. 
The symplectic form on ${\LL}$ extends to 
${\Lv}\times {\Lv}\to {\Q}$. 
This descends to a finite symplectic form 
\begin{equation*}
{\DL}\times {\DL}\to {\Q}/{\Z}, 
\end{equation*}
which we call the \textit{discriminant form} of ${\LL}$. 
The double dual relation ${\LL}^{\vee \vee}={\LL}$ assures that 
the discriminant form is \textit{nondegenerate}, i.e., 
$(x, {\DL})\not\equiv 0$ for any nonzero element $x$ of ${\DL}$. 
This property in turn implies that 
any homomorphism ${\DL}\to {\Q}/{\Z}$ is represented by 
the pairing with an element of ${\DL}$. 

The \textit{symplectic group} ${\SpL}$ of ${\LL}$ is the group of isometries of ${\LL}$. 
We define the finite symplectic group ${\rm Sp}({\DL})$ similarly. 
We denote by ${\GL}$ the kernel of the natural map ${\SpL}\to {\rm Sp}({\DL})$. 

A vector $v\in {\LL}$ is called \textit{primitive} if ${\Q}v\cap {\LL}={\Z}v$. 
For such a vector $v$, 
we denote by ${\dv}(v)$ the positive generator of the ideal $(v, {\LL})$ of ${\Z}$. 
The vector $v^{\ast}=v/{\dv}(v)$ is contained in ${\Lv}$ by definition. 
We denote by $[v^{\ast}]$ its class in ${\DL}$. 
By the primitivity of $v$, the order of $[v^{\ast}]$ as an element of the finite abelian group ${\DL}$ is equal to ${\dv}(v)$. 

In this section we prove the following.  

\begin{theorem}\label{thm: main}
Assume $d_1=d_2=1$. 
Then the map 
\begin{equation}\label{eqn: basic map}
\{ \: v\in {\LL} \: | \; \textrm{primitive} \: \} / {\GL} \longrightarrow {\DL}, \quad v\mapsto [v^{\ast}], 
\end{equation}
is bijective. 
\end{theorem} 

Note that the map \eqref{eqn: basic map} is well-defined (without assuming $d_1=d_2=1$) 
because ${\GL}$ acts on ${\DL}$ trivially by definition. 
The map \eqref{eqn: basic map} is clearly ${\SpL}$-equivariant. 
Therefore, for any subgroup $\Gamma$ of ${\SpL}$ containing ${\GL}$, 
the map \eqref{eqn: basic map} induces the bijective map 
\begin{equation*}
\{ \: v\in {\LL} \: | \; \textrm{primitive} \: \} / \Gamma \longrightarrow {\DL}/G 
\end{equation*}
where $G\simeq \Gamma/{\GL}$ is the image of $\Gamma \to {\rm Sp}({\DL})$. 

The basic line of the proof of Theorem \ref{thm: main} is similar to 
that of the orthogonal Eichler criterion (\cite{Sca}, \cite{FH}, \cite{GHS}). 
In the orthogonal case, a special type of isometries known as \textit{Eichler transvections} played a key role in the proof. 
We have the symplectic version of transvections as well,  
and they play a central role in our proof of Theorem \ref{thm: main}. 

More specifically, we first prove the surjectivity of \eqref{eqn: basic map} in \S \ref{ssec: surjective}. 
After recalling the symplectic transvections in \S \ref{ssec: transvection}, 
we deduce the injectivity of \eqref{eqn: basic map} in \S \ref{ssec: injective}. 
In \S \ref{ssec: (1,p,p)}, we prove that the map \eqref{eqn: basic map} is still bijective for 
$U\oplus U(p) \oplus \cdots \oplus U(p)$.

\subsection{Surjectivity}\label{ssec: surjective}

In this subsection we prove the following. 

\begin{proposition}\label{prop: surjective}
Assume that $d_1=1$. 
Then the map \eqref{eqn: basic map} is surjective. 
\end{proposition}

\begin{proof}
Let $x$ be an element of ${\DL}$. 
We write $d={\rm ord}(x)$. 
What has to be done is to find a primitive vector $v$ of ${\LL}$ with $[v^{\ast}]=x$. 

We write ${\LL}=U\oplus {\LL}'$ by the assumption $d_1=1$. 
We choose a vector $w\in ({\LL}')^{\vee}$ with 
$[w]=x$ in $D_{{\LL}'}={\DL}$. 
Then $dw\in {\LL}'$, but at this moment we do not know even whether it is primitive. 
So we modify $dw$ as follows. 
We choose a basis $e, f$ of $U$ with $(e, f)=1$. 
Then we put 
\begin{equation*}
v = de + dw \; \; \in  {\LL}. 
\end{equation*}
We prove that $v$ satisfies the desired properties. 

\begin{step}
$v$ is primitive in ${\LL}$. 
\end{step}

\begin{proof}
If $v/d'\in {\LL}$ for some natural number $d'>1$, 
then $(d/d')e\in U$ implies $d'|d$. 
If we denote $d''=d/d'$, 
we have $d''e+d''w\in {\LL}$. 
This implies $d''w\in {\LL}'$ and so 
${\rm ord}(x)\leq d'' < d$, 
which contradicts to ${\rm ord}(x)=d$. 
\end{proof}

\begin{step}
We have ${\dv}(v)=d$. 
\end{step}

\begin{proof}
Since $(v, f)=(de, f)=d$, 
we see that ${\dv}(v)|d$. 
On the other hand, we have 
$(v, {\LL})=d(e+w, {\LL})\subset d{\Z}$ 
because $e+w\in {\LL}^{\vee}$. 
This implies $d| {\dv}(v)$. 
\end{proof}

Finally, we have 
\begin{equation*}
[v^{\ast}] = [v/d] = [w+e] = [w] = x 
\end{equation*}
in ${\DL}$. 
This finishes the proof of Proposition \ref{prop: surjective}. 
\end{proof}

\subsection{Transvections}\label{ssec: transvection}

Let ${\LL}$ be a symplectic lattice. 
Let $l, m$ be vectors of ${\LQ}$ with $(l, m)=0$. 
We define a ${\Q}$-linear isomorphism 
$T_{l,m}\colon {\LQ} \to {\LQ}$ by 
\begin{equation*}
T_{l,m}(v) = v + (l, v)m + (m, v)l, \quad v\in {\LQ}. 
\end{equation*}
The condition $(l, m)=0$ ensures that $T_{l,m}\in {\rm Sp}({\LQ})$. 
We call $T_{l,m}$ a \textit{transvection}. 
This type of isometries have been classically considered. 
If we fix $l$ and run $m$, they generate the unipotent radical of the stabilizer of $l$ in ${\rm Sp}({\LQ})$. 
The relevance to the present paper comes from the following observation. 

\begin{lemma}\label{lem: transvection}
If $l, m\in {\LL}$, then $T_{l,m}\in {\GL}$. 
\end{lemma}

\begin{proof}
If $v\in {\LL}$, both $(l, v)$ and $(m, v)$ are integers, 
so $T_{l,m}(v)$ is contained in ${\LL}$. 
Hence $T_{l,m}\in {\SpL}$. 
Next, for $v\in {\Lv}$, $(l, v)$ and $(m, v)$ are still integers. 
This shows that $T_{l,m}(v)\equiv v$ mod ${\LL}$. 
Thus $T_{l,m}\in {\GL}$. 
\end{proof}

\subsection{Injectivity}\label{ssec: injective}

In this subsection we prove the injectivity of the map \eqref{eqn: basic map}. 
This amounts to the following. 

\begin{proposition}\label{prop: injective}
Assume $d_1=d_2=1$. 
Let $v, w$ be primitive vectors of ${\LL}$ with 
$[v^{\ast}]=[v^{\ast}]\in {\DL}$. 
Then $v$ and $w$ are ${\GL}$-equivalent. 
\end{proposition}

The proof is divided into several steps. 
We write 
\begin{equation}\label{eqn: 2U split} 
{\LL} = U_1 \oplus {\LL}' = U_1 \oplus U_2 \oplus {\LL}'' 
\end{equation}
by the assumption $d_1=d_2=1$, 
where $U_1$ and $U_2$ are copies of $U$. 
It would be convenient for later use to separate the first step in advance. 

\begin{lemma}\label{lem: save U}
Every vector of ${\LL}$ can be transformed by the ${\GL}$-action to a vector of ${\LL}'$. 
\end{lemma}

\begin{proof}
Let $v\in {\LL}$. 
We write $v=(v_1, v_2, v'')$ according to the decomposition \eqref{eqn: 2U split}. 
Since the unimodular lattice $U_{1} \oplus U_{2}$ has only one orbit of primitive vectors under 
${\rm Sp}(U_{1} \oplus U_{2})={\rm Sp}(4, {\Z})$, 
we can find an isometry $\gamma_0$ of $U_1\oplus U_2$ such that $\gamma_0(v_1, v_2)\in U_2$. 
If we put $\gamma = \gamma_{0}\oplus {\rm id}_{{\LL}''}$, 
then $\gamma\in {\GL}$ and $\gamma(v)\in {\LL}'$. 
\end{proof}

Actually Lemma \ref{lem: save U} will be the only place we use the assumption $d_2=1$. 
We now proceed to the proof of Proposition \ref{prop: injective}. 

\begin{proof}[(Proof of Proposition \ref{prop: injective})]
Let $v, w$ be primitive vectors of ${\LL}$ with $[v^{\ast}]=[w^{\ast}]$. 
By Lemma \ref{lem: save U}, we may assume 
\begin{equation}\label{eqn: in L'}
v, w\in {\LL}'. 
\end{equation}
We write 
\begin{equation*}
d = {\dv}(v) = {\rm ord}([v^{\ast}]) = {\rm ord}([w^{\ast}]) = {\dv}(w). 
\end{equation*}
We choose a basis $e, f$ of $U_1$ with $(e, f)=1$. 

\begin{step+}\label{step+1}
There exist $\gamma_1, \gamma_2\in {\GL}$ such that 
$\gamma_1(v) = v+de$ and $\gamma_2(w) = w+de$. 
\end{step+}

\begin{proof}
It suffices to prove this for $v$. 
Since $v\in {\LL}'$ and ${\dv}(v)=d$, 
there exists $v' \in {\LL}'$ such that $(v', v)=d$. 
If we put $\gamma_{1}=T_{e,v'}$, then 
\begin{equation*}
T_{e,v'}(v) = v + (v', v)e + (e, v)v' = v+de 
\end{equation*}
and $T_{e,v'}\in {\GL}$ by Lemma \ref{lem: transvection}. 
\end{proof}

\begin{step+}\label{step+2}
We have $d^2|(v, w)$. 
\end{step+}

\begin{proof}
By the assumption $[v^{\ast}]=[w^{\ast}]$, we have 
$v^{\ast}-w^{\ast}\in {\LL}$. 
Since $w^{\ast}\in{\Lv}$, we see that 
\begin{equation*}
(v^{\ast}-w^{\ast}, w^{\ast}) = (v^{\ast}, w^{\ast}) = d^{-2}(v, w) 
\end{equation*}
is contained in ${\Z}$. 
\end{proof}

\begin{step+}\label{step+3}
There exists $\gamma_3\in {\GL}$ such that 
\begin{equation*}
\gamma_3(v+de) = w + de + d\alpha f 
\end{equation*}
for some $\alpha\in {\Z}$. 
\end{step+}

\begin{proof}
We consider the vector $z=v^{\ast}-w^{\ast}\in {\LL}'$ as in Step \ref{step+2}. 
If we put $\gamma_{3}=T_{f,z}$, then 
$T_{f,z}\in {\GL}$ by Lemma \ref{lem: transvection} and 
\begin{eqnarray*}
T_{f,z}(v+de) 
& = & T_{f,z}(v) + dT_{f,z}(e) \\ 
& = & v+(z, v)f + d(e+(f, e)z) \\ 
& = & v - (w^{\ast}, v)f + de - d(v^{\ast} - w^{\ast}) \\ 
& = & w -d^{-1}(w, v)f + de. 
\end{eqnarray*}
Then $d^{-1}(w, v)\in d{\Z}$ by Step \ref{step+2}. 
\end{proof}

\begin{step+}\label{step+4}
There exists $\gamma_4\in {\rm Sp}(U_1)$ such that 
$\gamma_4(de+d\alpha f) = de$. 
\end{step+}

\begin{proof}
This holds because any two primitive vectors of $U_1$ 
($e+\alpha f$ and $e$ in this case) are equivalent under the action of ${\rm Sp}(U_1)\simeq {\rm SL}(2, {\Z})$. 
More explicitly, we may take 
$\gamma_{4} = T_{f, (\alpha/2)f}$. 
\end{proof}

We sum up the process so far.  
Starting form $v, w$ with $[v^{\ast}]=[w^{\ast}]$ and $v, w\in {\LL}'$, 
we have connected them as 
\begin{equation*}
v \stackrel{{\rm Step} 1}{\sim} 
v+de \stackrel{{\rm Step} 3}{\sim} 
w+ de +d\alpha f \stackrel{{\rm Step} 4}{\sim}
w+de \stackrel{{\rm Step} 1}{\sim}
w,  
\end{equation*}
where $\sim$ means ${\GL}$-equivalence. 
This proves Proposition \ref{prop: injective}. 
\end{proof}

By Proposition \ref{prop: surjective} and Proposition \ref{prop: injective}, 
the proof of Theorem \ref{thm: main} is now completed.

\subsection{The case of type $(1, p, \cdots, p)$}\label{ssec: (1,p,p)}

This subsection is a supplement to Theorem \ref{thm: main}. 
We give examples of symplectic lattices 
which do not satisfy $d_2=1$ but for which the Eichler criterion still holds. 
Our examples are generalization of \cite{HKW} Proposition I.3.38. 

\begin{proposition}\label{prop: (1,p,p)}
Let ${\LL}=U\oplus U(p) \oplus \cdots \oplus U(p)$ for a prime $p$. 
Then the conclusion of Theorem \ref{thm: main} is valid for ${\LL}$. 
\end{proposition}

\begin{proof}
Since Proposition \ref{prop: surjective} is valid when $d_1=1$, 
it suffices to verify only the injectivity part. 
We write ${\LL}=U\oplus {\LL}'$ where ${\LL}'=U(p)\oplus \cdots \oplus U(p)$. 
If $v\in {\LL}$ is a primitive vector, 
then ${\dv}(v)$ is the order of an element of the $p$-elementary abelian group ${\DL}$, 
so either ${\dv}(v)=1$ or ${\dv}(v)=p$. 
The bulk of the proof of Proposition \ref{prop: (1,p,p)} is the following assertion. 

\begin{claim}\label{claim: (1,p,p)}
Let $v\in {\LL}$ be a primitive vector. 

(1) If ${\dv}(v)=1$, then $v$ is ${\GL}$-equivalent to a vector of $U$. 

(2) If ${\dv}(v)=p$, then $v$ is ${\GL}$-equivalent to a vector of ${\LL}'$. 
\end{claim}

Proposition \ref{prop: (1,p,p)} can be deduced from Claim \ref{claim: (1,p,p)} as follows. 
When ${\dv}(v)={\dv}(w)=1$, 
then $v$ and $w$ are ${\GL}$-equivalent via primitive vectors of $U$ by the case (1) of Claim \ref{claim: (1,p,p)}. 
When ${\dv}(v)={\dv}(w)=p$, 
we may assume $v, w\in {\LL}'$ by the case (2) of Claim \ref{claim: (1,p,p)}. 
Then we can repeat the argument after \eqref{eqn: in L'} in the proof of Proposition \ref{prop: injective}, 
which uses only the first $U$ component. 

Now we prove Claim \ref{claim: (1,p,p)}. 
We write $v=(v_1, v_2, \cdots, v_g)$ according to the decomposition 
${\LL}=U\oplus U(p) \oplus \cdots \oplus U(p)$ 
where $v_1\in U$ and $v_i\in U(p)$ for $i>1$. 
When ${\dv}(v)=1$, then $v_1 \not\equiv 0$ mod $p$. 
The vector $(v_1, v_2)$ of $U\oplus U(p)$ is not necessarily primitive, 
but if we write $(v_1, v_2)=\alpha (v_{1}', v_{2}')$ with $(v_{1}', v_{2}')$ primitive, 
then $v_{1}' \not\equiv 0$ mod $p$. 
Thus $(v_{1}', v_{2}')$ is a ``short vector'' in the terminology of \cite{HKW} Definition I.3.36. 
Then we can apply \cite{HKW} Proposition I.3.38  
to find $\gamma\in \Gamma_{U\oplus U(p)}$ with $\gamma(v_1, v_2)\in U$. 
Repeating this process inductively for $v_2, v_3, \cdots, v_g$, 
we can reduce the number of $U(p)$ components successively and 
finally arrive at a vector of $U$. 
This proves the case (1) of Claim \ref{claim: (1,p,p)}. 

Next let ${\dv}(v)=p$. 
Then $v_1\equiv 0$ mod $p$. 
By the primitivity of $v$, at least one of $v_2, \cdots, v_g$, say $v_2$, 
satisfies $v_2\not\equiv 0$ mod $p$. 
If we write $(v_1, v_2)=\alpha (v_{1}', v_{2}')$ with $(v_{1}', v_{2}')$ primitive in $U\oplus U(p)$, 
then $\alpha \not\equiv 0$ mod $p$ and so 
$v_{1}' \equiv 0$ mod $p$ and $v_{2}' \not\equiv 0$ mod $p$. 
Thus $(v_{1}', v_{2}')$ is a ``long vector'' in the terminology of \cite{HKW} Definition I.3.36. 
Then, by \cite{HKW} Proposition I.3.38, 
the vector $(v_1, v_2)$ of $U\oplus U(p)$ can be transformed to a vector of $U(p)$ by the $\Gamma_{U\oplus U(p)}$-action. 
This proves the case (2) of Claim \ref{claim: (1,p,p)}. 
The proof of Proposition \ref{prop: (1,p,p)} is now finished. 
\end{proof}

\section{Splitting vectors}

Let ${\LL}$ be a symplectic lattice. 
A primitive vector $v\in {\LL}$ is called \textit{splitting} 
if it is contained in a rank $2$ sublattice ${\LL}_1$ with ${\LL}={\LL}_1\oplus {\LL}_{1}^{\perp}$. 
We must have ${\LL}_1\simeq U({\dv}(v))$. 
By writing ${\LL}_1$ as $\langle v, w \rangle$, 
this condition is equivalent to the existence of a primitive vector $w\in {\LL}$ satisfying 
${\dv}(w)={\dv}(v)$ and $(v, w)={\dv}(v)$. 
The purpose of this section is to classify splitting vectors by using the discriminant form. 

\subsection{The splitting criterion}

We first define the corresponding notion for elements of ${\DL}$. 

\begin{lemma}\label{lem: split}
Let $x$ be an element of ${\DL}$ of order $d$. 
The following conditions for $x$ are equivalent. 

(1) $x$ is contained in a subgroup $D_1$ of ${\DL}$ 
which is isomorphic to $({\Z}/d)^{\oplus 2}$ and satisfies ${\DL}=D_1 \oplus D_1^{\perp}$. 

(2) $x$ is contained in a nondegenerate subgroup $D_1$ of ${\DL}$ isomorphic to $({\Z}/d)^{\oplus 2}$. 

(3) There exists an element $y$ of ${\DL}$ such that 
${\DL}=\langle x, y \rangle \oplus \langle x, y \rangle^{\perp}$. 

(4) There exists an element $y$ of ${\DL}$ of order $d$ with $(x, y)=1/d\in {\Q}/{\Z}$. 
\end{lemma}

\begin{proof}
(1) $\Leftrightarrow$ (2): 
The property $D_1\cap D_{1}^{\perp} = \{ 0 \}$ means that $D_1$ is nondegenerate, 
and conversely, the nondegeneracy of $D_1$ implies that the decomposition ${\DL}=D_1 \oplus D_1^{\perp}$ holds. 

(1) $\Rightarrow$ (3) and (4) $\Rightarrow$ (2) are obvious. 



(3) $\Rightarrow$ (4): 
The nondegeneracy of $\langle x, y \rangle$ implies $(x, y)=k/d$ for some $k\in ({\Z}/d)^{\times}$. 
This in turn implies ${\rm ord}(y)=d$. 
Then we take $a\in ({\Z}/d)^{\times}$ with $ak = 1$ and replace $y$ by $ay$. 
\end{proof}

We call an element $x\in {\DL}$ \textit{splitting} if the conditions in Lemma \ref{lem: split} are satisfied. 
Our result is the following. 

\begin{theorem}\label{thm: splitting}
Assume $d_1=d_2=1$. 
A primitive vector $v$ of ${\LL}$ is splitting if and only if 
the element $[v^{\ast}]$ of ${\DL}$ is splitting. 
Hence the correspondence $v\mapsto [v^{\ast}]$ defines the bijective map 
\begin{equation*}
\{ \: v\in {\LL} \: | \; \textrm{splitting} \: \} / {\GL} \longrightarrow 
\{ \: x\in {\DL} \: | \; \textrm{splitting} \: \}.  
\end{equation*}
\end{theorem}

\begin{proof}
If $v$ is splitting, then $[v^{\ast}]$ is clearly splitting. 
We prove the converse. 
So suppose that $[v^{\ast}]$ is splitting. 
We denote $d={\dv}(v)={\rm ord}([v^{\ast}])$. 
We write ${\LL}=U\oplus {\LL}'$ by the assumption $d_1=1$. 
By Lemma \ref{lem: save U}, we may assume that $v\in {\LL}'$. 
Let $e, f$ be a basis of $U$ with $(e, f)=1$. 
By Step \ref{step+1} in the proof of Proposition \ref{prop: injective}, 
the vector $v$ is ${\GL}$-equivalent to $v+df$. 
Hence it is sufficient to prove that the vector $v+df$ is splitting. 

By the condition (4) in Lemma \ref{lem: split}, 
we can find an element $y$ of ${\DL}$ of order $d$ with $([v^{\ast}], y) = 1/d \in {\Q}/{\Z}$. 
Since ${\DL}=D_{{\LL}'}$, 
we can find a vector $w'$ of $({\LL}')^{\vee}$ such that $[w']=y$. 
Then $dw'\in {\LL}'$ because $y$ has order $d$. 
We modify $dw'$ as follows. 
Since $(v^{\ast}, w')\equiv 1/d$ mod ${\Z}$, 
we have $(v, dw')\equiv d$ mod $d^2{\Z}$. 
We write 
\begin{equation*}
(v, dw')=d+\alpha d^2, \quad \alpha\in {\Z}, 
\end{equation*}
and put 
\begin{equation*}
w = dw' + \alpha d e \; \; \in {\LL}. 
\end{equation*}

\begin{claim}\label{claim: split}
$w$ is a primitive vector of ${\LL}$ with ${\dv}(w)=d$ and $(v+df, w)=d$. 
\end{claim}

\begin{proof}
We first calculate 
\begin{equation*}
(v+df, w) = (v+df, \: dw' \! +\! \alpha de) = d + \alpha d^2 + \alpha d^2 (f, e) = d. 
\end{equation*}
Since $v+df$ is primitive with ${\dv}(v+df)=d$, 
the equality $(v+df, w)=d$ assures that $w$ is primitive as well and that 
${\dv}(w) | d$. 
On the other hand, we have 
\begin{equation*}
(w, {\LL}) = d(w'+\alpha e, {\LL}) \subset d{\Z} 
\end{equation*}
because $w'+\alpha e\in {\Lv}$. 
Hence $d| {\dv}(w)$. 
This concludes $d={\dv}(w)$.  
\end{proof}

Claim \ref{claim: split} implies that $v+df$ is splitting. 
This finishes the proof of Theorem \ref{thm: splitting}. 
\end{proof}

\subsection{Classification of splitting elements}\label{ssec: classify split}

Splitting elements of ${\DL}$ can be classified as follows. 
Let ${\DL}=\bigoplus_{p}{\DLp}$ be the decomposition into $p$-parts. 
Accordingly, we express an element of ${\DL}$ as $x=(x_{p})_{p}$ where $x_p\in {\DLp}$. 
By the condition (1) in Lemma \ref{lem: split}, we see that 
$x\in {\DL}$ is splitting if and only if $x_p\in {\DLp}$ is splitting for every prime $p$. 
Thus we shall focus on each $p$-component ${\DLp}$. 

We denote by $\mathcal{U}(p^{e})\simeq ({\Z}/p^{e})^{\oplus 2}$ 
the discriminant form of $U(p^{e})$. 
We have an orthogonal decomposition of the form  
\begin{equation*}
D_{\Lambda,p} \: = \: \bigoplus_{i=1}^{g} \mathcal{U}(p^{e_{p,i}}), \qquad 
0\leq e_{p,1} \leq e_{p,2} \leq \cdots \leq e_{p,g}. 
\end{equation*}
Accordingly, we express an element of ${\DLp}$ as 
$x_p=\sum_{i}x_{p,i}$ where $x_{p,i}\in \mathcal{U}(p^{e_{p,i}})$. 
We write 
${\rm ord}(x_{p,i})=p^{f_{p,i}}$ and 
${\rm ord}(x_{p})=p^{f_{p}}$. 
Then $f_{p,i}\leq e_{p,i}$ and $f_{p}=\max_{i} \{ f_{p,i} \}$. 

\begin{proposition}\label{prop: classify split}
A nonzero element $x_p=\sum_{i}x_{p,i}$ of ${\DLp}$ is splitting if and only if 
there is an index $i$ such that $f_{p,i}=f_{p}$ and $f_{p,i}=e_{p,i}$. 
\end{proposition}

\begin{proof}
If there is an index $i$ as above, 
the element $x_{p,i}\in \mathcal{U}(p^{e_{p,i}})$ has order $p^{e_{p,i}}$. 
Then we can find an element $y_{p}$ from $\mathcal{U}(p^{e_{p,i}})$  
such that 
${\rm ord}(y_{p}) = p^{e_{p,i}} = {\rm ord}(x_{p})$ 
and  
$(x_p, y_p) = (x_{p,i}, y_p) = p^{-e_{p,i}}$. 
This shows that $x_p$ is splitting. 

Conversely, suppose that 
$f_{p,i}<e_{p,i}$ for any index $i$ with $f_{p,i}=f_{p}$. 
Then, for every $1\leq i \leq g$, we have 
$(x_{p,i}, y_{p,i})\in p^{-f_{p}+1}{\Z}/{\Z}$ 
for any element $y_{p,i}$ of $\mathcal{U}(p^{e_{p,i}})$ of order $\leq p^{f_{p}}$. 
(Consider the case $f_{p,i}=f_{p}<e_{p,i}$ and the case $f_{p,i}<f_{p}$ separately.) 
It follows that 
$(x_p, y_p)\in p^{-f_{p}+1}{\Z}/{\Z}$ 
for any element $y_{p}$ of ${\DLp}$ of order $p^{f_{p}}$. 
Hence $x_{p}$ is not splitting. 
\end{proof}

\begin{corollary}
Assume $d_1=d_2=1$. 
Every primitive vector of ${\LL}$ is splitting if and only if  
$d_3, d_4/d_3, \cdots, d_g/d_{g-1}$ are square-free and coprime to each other. 
\end{corollary}

\begin{proof}
By Proposition \ref{prop: classify split}, 
every element of ${\DL}$ is splitting 
if and only if $e_{p,g}\leq 1$ for all prime $p$. 
In this case, there is at most one index $i(p)$ for each $p$ such that 
$e_{p, i(p)}=0$ and $e_{p, i(p)+1}=1$; 
we have $e_{p,i}=e_{p,i+1}$ for other indices $i$. 
Since $d_i=\prod_{p}p^{e_{p,i}}$, 
we have $d_{i+1}/d_{i} = \prod_{p} p^{e_{p,i+1}-e_{p,i}}$. 
Therefore the above condition is equivalent to the property that 
$d_3/d_2, d_4/d_3, \cdots, d_{g}/d_{g-1}$ 
are square-free and coprime to each other. 
\end{proof}

\begin{remark}
The square-free and coprime conditions are also considered by Schellhammer (\cite{Sch} Definition 2.2). 
\end{remark}

\begin{remark}
Every primitive vector is splitting also for $U\oplus U(p) \oplus \cdots \oplus U(p)$. 
In this case, when ${\dv}(v)=1$, $v$ is obviously splitting. 
When ${\dv}(v)=p$, we may assume $v\in U(p) \oplus \cdots \oplus U(p)$ 
by Claim \ref{claim: (1,p,p)} (2). 
Then $v$ is splitting in $U(p) \oplus \cdots \oplus U(p)$. 
\end{remark}


\end{document}